\documentclass[12pt,a4paper,twoside,english]{smfart}

\usepackage{lmodern}
\usepackage[utf8]{inputenc}
\usepackage[francais]{babel}
\usepackage[T1]{fontenc}
\usepackage{amssymb}
\usepackage{amsmath}
\usepackage{amsthm}
\usepackage{amsfonts}
\usepackage{graphicx}
\usepackage[colorlinks=true,linkcolor=black,citecolor=black,urlcolor=black]{hyperref}
\usepackage{fancyhdr}
\usepackage[all,cmtip]{xy}
\usepackage{alltt}
\usepackage{footmisc}
\usepackage{smfthm}

\usepackage{calrsfs}

\usepackage{fullpage}

\xyoption{all}

\xyoption{frame}

\newcommand{\finpreuve}{\mbox{} \hfill \mbox{$\Box$}}

\DeclareUnicodeCharacter{00A0}{~}

\def\Q{{\mathbb Q}}
\def\Z{{\mathbb Z}}
\def\fq{{\mathbb F}}

\def\Nat{{\mathbb N}}

\def\p{{\mathfrak p}}
\def\P{{\mathfrak P}}

\def\q{{\mathfrak q}}

\def\aa{{\mathfrak  a}}

\def\X{{\rm X}}

\def\R{{\rm R}}
\def\d{{\rm d}}
\def\G{{\rm G}^{ur}}
\def\K{{\rm K}}
\def\L{{\rm L}}

\def\k{{\rm k}}

\def\W{{\rm W}}
\def\GG{{\rm G}^{ur}}

\def\V{{\rm V}}

\def\A{{\rm A}}
\def\spec{{\rm Spec }}
\def\Gg{{\rm G}}

\def\M{{\rm M}}

\def\FM{{\rm FM}}

\def\H{{\mathcal H}}

\def\O{{\mathcal O}}

\def\SS{{\mathcal S}}

\def\B{{\mathcal B}}

\def\Ff{{\mathcal F}}

\def\rk{\rm rk}

\def\Gl{{\rm Gl}}

\def\Gal{{\rm Gal}}
\def\rg{{\rm rg}}

\def\Cl{{\rm Cl}}

\def\Sl{{\rm Sl}}

\def\Rad{{\mathcal R}ad_r}
\def\sl{{\mathfrak s} {\mathfrak l}}
\def\1{{\bf 1}}
\def\disc{{\rm disc}}
\def\S{{\rm S}}

\newtheorem*{Theorem}{Theorem}
\newtheorem*{conjectureNC}{Conjecture}
\newtheorem{conjecture}{Conjecture}
\newtheorem{Corollary}{Corollary}
\newtheorem*{Remark}{Remark}

\begin{document}

\date{\today}


\author{Christian Maire}

\address{Laboratoire de Mathématiques de Besançon (UMR 6623), Université Bourgogne Franche-Comté et CNRS, 16 route de Gray, 25030 Besançon cédex, France}
\address{Department of Mathematics, 310 Malott Hall, Cornell University, Ithaca, NY USA 14853}
\email{christian.maire@univ-fcomte.fr}
  
\title{Unramified $2$-extensions of totally imaginary number fields and $2$-adic analytic groups}



\thanks{\emph{Acknowledgements.} 
This work has been done during a visiting scholar position  at Cornell University for  the academic year 2017-18, and funded by the program "Mobilit\'e sortante" of the R\'egion Bourgogne Franche-Comt\'e.  The author  thanks the  Department of Mathematics at Cornell University for providing a stimulating research atmosphere. He also thanks Georges Gras and Farshid Hajir for the encouragements and their  useful remarks, and  Ravi Ramakrishna for very inspiring discussions.
}

\begin{abstract} Let $\K$ be a totally imaginary number field. Denote by $\G_\K(2)$ the Galois group  of the maximal unramified pro-$2$ extension of $\K$.  
By comparing cup-products in \'etale cohomology of $\spec \O_\K$ and cohomology  of uniform pro-$2$ groups, we obtain  situations where $\G_\K(2)$ has no 
non-trivial uniform analytic quotient,  proving some new special cases of the unramified  Fontaine-Mazur conjecture. For example, in the family of imaginary quadratic  fields $\K$ for which the $2$-rank of the class group 
is equal to~$5$, we obtain that for at  least $33.12 \% $ of such  $\K$, the group $\G_\K(2)$ has no  non-trivial uniform analytic quotient. 
\end{abstract}

\maketitle

\tableofcontents

\section*{Introduction}

Given a number field $\K$ and a prime number $p$,
the tame version of the  conjecture  of Fontaine-Mazur  (conjecture (5a) of \cite{FM}) asserts that every  finitely and tamely ramified  continuous
Galois representation  $\rho : \Gal(\overline{\K}/\K) \rightarrow \Gl_m(\Q_p)$ of the absolute Galois group of $\K$, has  finite image. 
Let us mention briefly two strategies to attack  this conjecture.

$-$ The first one is to use the techniques coming from the considerations that inspired the conjecture, {\em i.e.}, from the Langlands program 
(geometric Galois representations, modular forms, deformation theory, etc.). Fore more than a decade, many authors have contributed to understanding 
the foundations of this conjecture with some serious progress having been made. As a partial list of such results, we refer the reader to Buzzard-Taylor \cite{BT},  
Buzzard \cite{Buzzard}, Kessaei \cite{Kassaei1}, Kisin \cite{Kisin}, Pilloni \cite{Pilloni}, Pilloni-Stroh \cite{Pilloni-Stroh}, etc.

$-$ The second one consists in comparing  properties of $p$-adic analytic pro-$p$ groups and arithmetic.  
Thanks to this strategy, Boston gave  in the 90's the first evidence for the tame version of the Fontaine-Mazur conjecture (see \cite{Boston1}, \cite{Boston2}).  This one has been extended by Wingberg \cite{Wingberg}. See also \cite{Maire} and  the recent work of  Hajir-Maire  \cite{H-M}. In all of these situations, the key fact is to use a semi-simple action. Typically, this approach  gives no information for quadratic number fields when $p=2$.

\

In this work, when $p=2$, we propose to give some families of  imaginary quadratic number fields 
for which the unramified Fontaine-Mazur conjecture is true (conjecture (5b) of \cite{FM}). For doing this, we compare the étale cohomology of $\spec \O_\K$ and the cohomology of $p$-adic analytic pro-$p$ groups. In particular, we exploit the fact that in characteristic $2$, cup products in  $H^2$ could be not alternating (meaning $x\cup x  \neq 0$), more specificaly, a beautiful computation of cup-products in $ H_{et}^3(\spec \O_k, \fq_2)$ made by Carlson and Schlank in \cite{Carlson-Schlank}.
And so, surprisingly our strategy works \emph{only} for $p=2$ !

\

Given a prime number $p$, denote by $\K^{ur}(p)$ the maximal pro-$p$ unramifed extension of $\K$; put $\G_\K(p):=\Gal(\K^{ur}(p)/\K)$.  Here we are interested in  uniform   quotients of $\G_\K(p)$ (see section \ref{section2} for  definition) which are related to
the unramified Fontaine-Mazur conjecture thanks to the following equivalent version:

\begin{conjecture} \label{conj2}
Every uniform quotient $\Gg$ of $\G_\K(p)$ is trivial.
\end{conjecture} 


Remark that Conjecture \ref{conj2} can be rephrased as follows: the pro-$p$ grop $\G_\K(p)$ has no  uniform quotient $\Gg$ of dimension $d$ for all $d>0$.
Of course, this is obvious when $d>\d_2 \Cl_\K$, and when $d\leq 2$, thanks to the fact that  $\G_\K(p)$ is FAb (see later the definition).

\medskip

Now take $p=2$. Let $(x_i)_{i=1,\cdots, n}$ be an  $\fq_2$-basis  of  $H^1(\G_\K(2),\fq_2)\simeq H_{et}^1(\spec \O_\k,\fq_2)$, and consider the $n\times n$-square matrix $\M_\K:=(a_{i,j})_{i,j}$ with coefficients in $\fq_2$, where $$a_{i,j}=x_i\cup x_i \cup x_j,$$
thanks to the fact that here $H^3_{et}(\spec \O_\K,\fq_2) \simeq \fq_2$.
As we will see, this is   the Gram matrix of a certain bilinear form defined, via Artin symbol, on the Kummer radical of
the $2$-elementary abelian maximal unramified extension $\K^{ur,2}/\K$ of $\K$. We also will see that for imaginary quadratic number fields, this matrix is often of large rank.

First, we prove:

\begin{Theorem} \label{maintheorem0} 
Let $\K$ be a totally imaginary number field. Let $n:=d_2\Cl_\K$ be the $2$-rank of the class group of $\K$. 
\begin{itemize}
\item[$(i)$] Then the pro-$2$ group $\G_\K(2)$ has no  uniform  quotient of dimension $d>n-\frac{1}{2} \rk(\M_\K)$.
\item[$(ii)$] Moreover,   Conjecture \ref{conj2} holds (for $p=2$) when:
\begin{enumerate}
\item[$\bullet$]  $n=3$, and $\rk(\M_\K)>0$;
\item[$\bullet$]  $n=4$, and $\rk(\M_\K) \geq 3$;
\item[$\bullet$]   $n=5$, and $\rk(\M_K)=5$.
\end{enumerate}
\end{itemize}
\end{Theorem}

By relating the matrix $\M_\K$ to a R\'edei-matrix type, and thanks to the  work of Gerth \cite{Gerth} and Fouvry-Kl\"uners \cite{Fouvry-Klueners}, one  can also deduce some density information when $\K$ varies in the family $\Ff$ of  imaginary quadratic  fields.
For $n,d, X \geq 0$, denote by $$\S_X:=\{ \K \in \Ff,  \ -\disc_\K \leq X\}, \ \ \S_{n,X}:=\{ \K\in \S_{X}, \ \ d_2 \Cl_\K=n\}$$
$$\FM_{n,X}^{(d)}:=\{ \K \in \S_{n,X},  \ \G_\K(2) {\rm \ has \ no \ uniform \ quotient \ of \ dimension} > d\},$$
 $$ \FM_{n,X}:=\{ \K \in \S_{n,X}, \  {\rm Conjecture \ \ref{conj2} \ holds  \ for \ }\K\},$$
 and consider the limits:
$$\FM_n^{(d)}:= \liminf_{X\rightarrow + \infty} \frac{\# \FM_{n,X}^{(d)}}{\#\S_{n,X}}, \ \ \ \FM_{n}:=\liminf_{X \rightarrow + \infty} \frac{\# \FM_{n,X}}{\#\S_{n,X}}.$$
$\FM_n$ measures the proportion of  imaginary quadratic fields $\K$ with $d_2 \Cl_\K=n$, for which  Conjecture \ref{conj2} holds (for $p=2$); and $\FM_n^{(d)}$ measures the proportion of imaginary quadratic fields  $\K$  with  $d_2 \Cl_\K=n$, for which  $\G_\K(2)$ has no uniform quotient of dimension $>d$.

\medskip

Then \cite{Gerth} allows us to obtain the following densities for  uniform groups of small dimension:

\begin{Corollary} \label{coro-intro1}
One has: 
\begin{enumerate}
\item[$(i)$] $\FM_3 \geq .992187$,
\item[$(ii)$]  $\FM_4 \geq .874268$, $\FM_4^{(4)} \geq .999695$,
\item[$(iii)$] $\FM_5 \geq .331299$, $\FM_5^{(4)} \geq .990624$, $\FM_5^{(5)} \geq .9999943$,
\item[(iv)] for all $d\geq 3$, $\FM_{d}^{(1+d/2)} \geq 0.866364 $, and $\FM_d^{(2+d/2)} \geq .999953$.
\end{enumerate}
\end{Corollary}

\medskip

\begin{Remark} At this level, one should make two observations. 

1) Perhaps for many $\K \in \S_{3,X}$ and $\S_{4,X}$, the pro-$2$ group $\G_\K(2)$ is finite but, by the Theorem of Golod-Shafarevich (see for example \cite{Koch}), for every $\K \in \S_{n,X}$, $n\geq 5$,
the pro-$2$ group $\G_\K(2)$ is infinite. 

2) In our work, it will appear that  we have no information about the  Conjecture \ref{conj2} for  number fields $\K$ for which  
the $4$-rank of the class group is large.  Typically,   in $\FM_i$ one keeps out all the number fields having   maximal $4$-rank. 
\end{Remark}

To conclude, let us mention a general asymptotic estimate thanks to  the  work  of Fouvry-Kl\"uners \cite{Fouvry-Klueners}.
Put $$\FM_{X}^{[i]}:=\{\K \in \S_{X},  \ \G_\K(2) {\rm \ has \ no \ uniform \ quotient \ of \ dimension} > i+ \frac{1}{2}d_2 \Cl_\K\}$$
and $$\FM^{\lbrack i\rbrack }:= \liminf_{X\rightarrow + \infty} \frac{\# \FM^{[i]}_X}{\# \S_{X}}.$$

Our work allows us to obtain:
\begin{Corollary} \label{coro-intro2}
 One has: $$\FM^{[1]} \geq .0288788, \ \ \FM^{[2]} \geq 0.994714,  \ \ {\rm and } \ \ \FM^{[3]} \geq 1-9.7 \cdot 10^{-8}.$$
\end{Corollary}

\

\

This paper has three sections. 
In Section 1 and Section 2, we give the basic tools  concerning the \'etale cohomology of number fields and the $p$-adic analytic groups. 
Section 3 is devoted to  arithmetic considerations. After the presentation of our strategy, we develop some basic facts about bilinear forms over $\fq_2$, specially for 
the  form introduced in our study (which is defined on a certain Kummer radical). In particular, we insist on the role played by totally isotropic subspaces. To finish, we consider a relation with a R\'edei matrix  that allows us to obtain  density information.

\


 
 

\

\medskip

{\bf Notations.}

Let $p$ be a prime number and let $\K$ be a number field. 
 
Denote by 
\begin{enumerate}
\item[$-$]  $p^*=(-1)^{(p-1)/2}p$, when $p$ is odd;
\item[$-$] $\O_\K$ the ring of integers of $\K$;
\item[$-$] $\Cl_\K$ the $p$-Sylow of the Class group of $\O_\K$;
\item[$-$] $\K^{ur}$ the maximal profinite extension of $\K$ unramified everywhere. Put $\GG_\K=\Gal(\K^{ur}/\K)$;
\item[$-$] $\K^{ur}(p)$ the maximal pro-$p$ extension of $\K$ unramified everywhere. Put $\G_{\K}(p):=\Gal(\K^{ur}(p)/\K)$;
\item[$-$] $\K^{ur,p}$ the elementary abelian maximal unramified $p$-extension of $\K$.
\end{enumerate}
Recall that the group $\G_{\K}(p)$ is a  finitely presented pro-$p$ group. See \cite{Koch}. See also \cite{NSW} or \cite{Gras}. 
Moreover by class field theory, $\Cl_\K $ is isomorphic to the abelianization of $\G_\K(p)$. In particular it implies that every open subgroup $\H$ of $\G_{\K}(p)$ has 
finite abelianization:  this property is kwnon as  "FAb". 

\

\section{Etale cohomology: what we need} \label{section1}

\subsection{} 
For what follows,  the references are plentiful: \cite{Mazur}, \cite{Milne}, \cite{Milne1}, \cite{Schmidt1}, \cite{Schmidt2}, etc. 

\medskip

Assume  that $\K$ is totally imaginary when $p=2$, and put $\X_\K= \spec \O_\K$. 

The Hochschild-Serre spectral sequence (see \cite{Milne}) gives for every $i\geq 1$ a map $$\alpha_i : H^i (\G_\K(p)) \longrightarrow H^i_{et}(\X_\K), $$
where the coefficients are in $\fq_p$ (meaning the constant sheaf for the \'etale site $\X_\K$).
As $\alpha_1$ is an isomorphism, one obtains the long exact sequence:
$$H^2(\G_\K(p)) \hookrightarrow H^2_{et}(\X_\K) \longrightarrow H_{et}^{2}(\X_{\K^{ur}(p)}) \longrightarrow
H^3(\G_\K(p)) \longrightarrow H^3_{et}(\X_\k) $$
where $H^3_{et}(\X_\k) \simeq (\mu_{\K,p})^\vee$, here $(\mu_{\K,p})^\vee$  is the Pontryagin dual of the  group of $p$th-roots of unity in $\K$.

\medskip

\subsection{}
Take now $p=2$. 
Let us give  $x,y,z \in H^1_{et}(\X_\K)$.
In  \cite{Carlson-Schlank}, Carlson and Schlank give a formula in order to determine   the cup-product $x \cup y \cup z \in H^3_{et}(\X_\K)$. In particular, they
show how to produce some arithmetical situations for which such cup-products 
$x\cup x \cup y $ are not zero. 
Now,  one has the commutative diagram:
$$\xymatrix{H^3(\G_\K(p)) \ar[r]^{\alpha_3} & H^3_{et}(\X_\K) \\
H^1(\G_\K(p))^{\otimes^3}\ar[r]^\simeq_{\alpha_1} \ar[u]^{\beta} &  \ar[u]^{\beta_{et}} H^1_{et}(\X_\K)^{\otimes^3}}$$  
Hence $(\alpha_3\circ \beta)(a\otimes b \otimes a)=\alpha_1(a)\cup \alpha_1(b)\cup \alpha_1(c)$. By taking $x=\alpha_1(a)=\alpha_1(b)$ and $y=\alpha_1(c)$, one gets 
$a\cup a \cup b \neq 0 \in H^3(\G_\K(p))$ when $x\cup x\cup  y \neq 0 \in H^3_{et}(\X_\K)$.

\subsection{The computation of Carlson and Schlank} \label{section:Carlson-Schlank}
Take $x$ and $y$ two non-trivial characters of $H^1(\G(p))\simeq H^1_{et}(\X)$. Put $\K_x=\K^{ker(x)}$ and $\K_y=\K^{ker(y)}$. By Kummer theory, there exist $a_x, a_y \in \K^\times/(\K^\times)^2$ such that $\K_x=\K(\sqrt{a_x})$ and $\K_y=\K(\sqrt{a_y})$.
As the extension $\K_y/\K$ is  unramified, for every prime ideal $\p$ of $\O_\K$, the $\p$-valuation $v_\p(a_y)$ is even, and then $\sqrt{(a_y)}$ has a sense  (as an ideal of $\O_\K$).
Let us write $$\sqrt{(a_y)}:=\prod_i\p_{y,i}^{e_{y,i}}.$$

Denote by $I_x$ the set of prime ideals $\p$ of $\O_\K$ such that  $\p$ is inert in $\K_x/\K$ (or equivalently, $I_x$ is the set of primes of $\K$ such that the Frobenius at $\p$ generates $\Gal(\K_x/\K)$). 

\begin{prop}[Carlson and Schlank] \label{proposition:C-S}
The cup-product $x\cup x \cup y \in H^3_{et}(X)$ is non-zero if and only if, $\displaystyle{\sum_{\p_{y,i} \in I_x}e_{y,i}}$ is odd.
\end{prop}

\begin{rema} \label{remarque:symbole} The condition  of Proposition \ref{proposition:C-S} is equivalent to the triviality of the  Artin symbol $\displaystyle{\left(\frac{\K_x/\K}{\sqrt{(a_y)}}\right)}$.
Hence if one takes $b_y=a_y \alpha^2$ with $\alpha\in \K$ instead of $a_y$ then, as
 $\displaystyle{\left(\frac{\K_x/\K}{(\alpha)}\right)}$ is trivial, the condition is well-defined.
\end{rema}

Let us give an easy example inspired by a computation of \cite{Carlson-Schlank}.

\begin{prop}\label{criteria}
Let $\K/\Q$ be an imaginary quadratic field. Suppose that there exist $p$ and $q$ two different odd prime numbers ramified in $\K/\Q$, and  such that: $\displaystyle{\left(\frac{p^*}{q}\right)=-1}$.
Then there exist $x \neq y \in H^1_{et}(\X_\K)$ such that $x\cup x \cup y \neq 0$.
\end{prop}

\begin{proof} 
Take $\K_x=\K(\sqrt{p^*})$ and $\K_y=\K(\sqrt{q^*})$, and apply Proposition \ref{proposition:C-S}.
\end{proof}

\section{Uniform pro-$p$-groups and arithmetic: what we need} \label{section2}

\subsection{} Let us start with the definition of a uniform pro-$p$ group (see for example \cite{DSMN}).

\begin{defi}
Let $\Gg$ be a finitely generated pro-$p$ group. 
We say that $\Gg$ is uniform if:
\begin{enumerate}
\item[$-$] $\Gg$ is torsion free, and
\item [$-$] $[\Gg,\Gg] \subset \Gg^{2p}$. 
\end{enumerate}
\end{defi}

\begin{rema}
For a uniform  group $\Gg$, the $p$-rank of $\Gg$ coincides with  the dimension of $\Gg$. 
\end{rema}

The uniform pro-$p$ groups play a central rule in the study of analytic pro-$p$ group, indeed:

\begin{theo}[Lazard \cite{lazard}] \label{Lazard0}
Let $\Gg$ be a profinite group. Then $\Gg$ is $p$-adic analytic {\em i.e.} $\Gg \hookrightarrow_c {\rm Gl}_m(\Z_p)$ for a certain positive integer $m$, if and only if, $\Gg$ contains an open  uniform subgroup $\H$. 
\end{theo}

\begin{rema} 
For different equivalent definitions of $p$-adic analytic groups, see \cite{DSMN}. See also  \cite{Lubotzky-Mann}.
\end{rema}

\begin{exem}
The correspondence between $p$-adic analytic pro-$p$ groups and $\Z_p$-Lie algebra via the log/exp maps, 
allows to give examples of uniform pro-$p$ groups (see \cite{DSMN}, see also \cite{H-M}). Typically,
let $\sl_n(\Q_p)$ be the $\Q_p$-Lie algebra of  the square matrices  $n\times n$ with  coefficients in  $\Q_p$ and of zero trace. It is a simple algebra of dimension $n^2-1$. 
Take the natural basis:
\begin{enumerate} 
\item[(a)] for $i \neq j$, $E_{i,j}=(e_{k,l})_{k,l}$ for which all the coefficient are zero excepted  $e_{i,j}$ that takes value $2p$;
\item[(b)] for $i>1$, $D_i=(d_{k,l})_{k,l}$ which is the diagonal matrix $D_i=(2p,0,\cdots,0,-2p,0,\cdots, 0)$, where $d_{i,i}=-2p$. 
\end{enumerate}
Let $\sl_n$ be the $\Z_p$-Lie algebra generated by the $ E_{i,j}$ and the $D_i$. 
Put $X_{i,j}=\exp E_{i,j}$ and $Y_i=\exp D_i$. Denote by $\Sl_n^1(\Z_p)$ the subgroup of $\Sl_n(\Z_p)$ generated by the matrices $X_{i,j}$ and $Y_i$. The group $\Sl_n^1(\Z_p)$ 
is uniform  and of  dimension $n^2-1$. It is also the kernel of the reduction map  of  $\Sl_n(\Z_p)$ modulo $2p$. Moreover, $\Sl_n^1(\Z_p)$ is also FAb, meaning that every open subgroup $\H$ has finite abelianization.
 \end{exem}
 
\medskip
 
Recall  by Lazard \cite{lazard} (see also \cite{Symonds-Weigel} for an alternative proof): 
\begin{theo}[Lazard \cite{lazard}] \label{Lazard} Let $\Gg$ be a uniform pro-$p$ group (of dimension $d>0$). Then for all $i\geq 1$, one has: $$H^i(\Gg) \simeq \bigwedge^i(H^1(\Gg),$$
where here the exterior product is induced by the cup-product.
\end{theo}

As consequence, one has immediately:

\begin{coro} Let $\Gg$ be a uniform pro-$p$ group. Then  for all $x,y \in H^1(\Gg)$, one has $x\cup x \cup y =0 \in H^3(\Gg)$.
\end{coro}

\begin{rema}
For $p>2$, Theorem \ref{Lazard} is an equivalence: a pro-$p$ group $\Gg$ is uniform if and only if,  for $i\geq 1$, $H^i(\Gg) \simeq \bigwedge^i(H^1(\Gg)$. (See \cite{Symonds-Weigel}.)
\end{rema}

Let us mention  another consequence useful in our context:

\begin{coro}
Let $\Gg$ be a FAb uniform prop-$p$ group of dimension $d>0$. Then $d\geq 3$.
\end{coro}

\begin{proof}
Indeed, if $\dim \Gg=1$, then $\Gg \simeq \Z_p$ ($\Gg$ is pro-$p$ free) and, if $\dim\Gg=2$, then by Theorem  \ref{Lazard0}, $H^2(\Gg) \simeq \fq_p$ and $\Gg^{ab} \twoheadrightarrow \Z_p$. Hence, $\dim \Gg $ should be at least $3$.
\end{proof}

\subsection{}

Let us recall  the Fontaine-Mazur conjecture (5b) of \cite{FM}. 

\begin{conjectureNC} \label{conj1}
Let $\K$ be a number field. Then every continuous Galois representation $\rho : \G_\K \rightarrow \Gl_m(\Z_p)$ has finite image.
\end{conjectureNC}

Following the result of Theorem \ref{Lazard0} of Lazard, we see that proving  Conjecture (5b) of \cite{FM} 
for $\K$, is equivalent to proving Conjecture \ref{conj2} for every finite extension $\L/\K$ in $\K^{ur}/\K$. 

 

\section{Arithmetic consequences}

\subsection{The strategy} \label{section:strategy}
Usually, when $p$ is odd,  cup-products factor through  the exterior product. But, for $p=2$, it is not the case! This is the obvious observation that we will combine with  \'etale cohomology and with the  cohomology of uniform pro-$p$ groups.

\

From now on we assume that $p=2$.

\

Suppose  given $\Gg$  a non-trivial uniform  quotient of $\G_\K(p)$. Then by the inflation map one has: $$H^1(\Gg) \hookrightarrow H^1(\G_\K(p)).$$ Now take  $a,b \in H^1(\G_\K(p))$ coming from $H^1(\Gg)$.
Then, the cup-product  $a\cup a \cup  b \in H^3(\G_\K(p))$ comes  from  $H^3(\Gg)$ by the inflation map. In other words, one has the following commutative diagram:
$$\xymatrix{H^3(\Gg) \ar[r]^{inf}&H^3(\G_\K(p)) \ar[r]^{\alpha_3} & H^3_{et}(\X_\K) \\
H^1(\Gg)^{\otimes^3} \ar@{->>}[u]^{\beta_0} \ar@{^(->}[r] &H^1(\G_\K(p))^{\otimes ^3}\ar[r]^\simeq \ar[u]^\beta &  \ar[u]^{\beta_{et}} H^1_{et}(\X_\K)^{\otimes^3} }$$ 
But by Lazard's result (Theorem \ref{Lazard}),  $\beta_0(a\otimes a \otimes b)=0$, and then one gets a contradiction if $\alpha_1(a)\cup \alpha_1(a) \cup \alpha_1(b) $ is non-zero in $H^3_{et}(\X_\K)$: it is at this level that one may use the computation of Carlson-Schlank.

\medskip

Before  developing this observation in the context of  analytic pro-$2$ group, 
 let us give two immediate consequences: 

\begin{coro}
Let $\K/\Q$ be a quadratic imaginary number field  satisfying the condition of Proposition \ref{criteria}. Then $\G_{\K}(2)$ is of cohomological dimension at least $3$.
\end{coro}

\begin{proof}
Indeed, there exists a non-trivial cup-product $x\cup x \cup y \in H^3_{et}(X)$ and then non-trivial in $H^3(\Gg_\K(2))$.
\end{proof}

\medskip

\begin{coro}\label{coro-exemple} Let $p_1,p_2,p_3,p_4$ be four  prime numbers  such that $p_1 p_2 p_3 p_4 \equiv 3 (\mod 4)$.
Take $\K=\Q(\sqrt{-p_1  p_2  p_3 p_4})$. Suppose  that there exist $i\neq j$ such that $\displaystyle{\left(\frac{p_{i}^*}{p_{j}}\right)=-1}$. Then $\G_\K(2)$ has non-trivial uniform quotient.
\end{coro}

\begin{rema} Here, one may replace $p_1$ by $2$. 
But we are not guaranteed in all cases of the infiniteness of $\G_\K(2)$, as we are outside the conditions of the result of Hajir  \cite{Hajir}. We will see later the reason.
\end{rema}

\begin{proof}
Let us start with a non-trivial uniform quotient $\Gg$ of $\G_\K(2)$. As by class field theory, the pro-$2$ group $\Gg$ should be FAb, it is of dimension $3$, {\emph i.e.} $H^1(\Gg) \simeq H^1(\G_\K(2))$. By Proposition \ref{criteria}, there exist $x,y \in H^1(\Gg)$ such that $x\cup y \neq 0 \in H^3_{et}(\X_\K)$, and the "strategy" applies. 
\end{proof}

Now, we would like to extend this last construction.

\subsection{Bilinear forms over $\fq_2$ and conjecture \ref{conj2}}

\subsubsection{Totally isotropic subspaces}

Let $\B$ be a bilinear form over an $\fq_2$-vector space $\V$ of finite dimension. Denote by $n$ the dimension of $\V$ and by $\rk(\B)$ the rank of $\B$. 

\begin{defi}
Given a bilinear form $\B$, one define the index $\nu(\B)$ of $\B$ by $$\nu(\B):=\max \{ \dim W, \ \B(W,W)=0\}.$$
\end{defi}

The index $\nu(\K)$ is then an upper bound of the dimension of totally isotropic subspaces $W$ of $\V$. As we will see, the index $\nu(\B)$ is well-known when $\B$ is symmetric.
For the general case, one has:

\begin{prop} \label{bound-nu} The index $\nu(\B)$ of a bilinear form $\B$ is at most  than  $n - \frac{1}{2}\rk(\B)$.
\end{prop}
 
 \begin{proof}
  Let $\W$ be a totally isotropic subspace of $\V$ of dimension $i$. Let us complete a basis of $W$ to a basis ${\rm B}$ of $\V$. It  is then easy to see  that the Gram
  matrix of $\B$ in ${\rm B}$ is of rank  at most $2n-2i$. 
 \end{proof}
 
 This bound is in a  certain sense optimal as we can  achieve it in the symmetric case.

\begin{defi} 

$(i)$ Given $a\in \fq_2$.
The bilinear form $b(a)$ with matrix $\left(\begin{array}{cc} a&1 \\
1&0 \end{array}\right) $ is called a metabolic plan.
A metabolic form is an orthogonal sum of metabolic plans (up to isometry).

$(ii)$ A symmetric bilinear form $(V,\B)$ is called alternating if $\B(x,x) = 0$  for all $x\in V$. Otherwise $\B$ is called nonalternating.
\end{defi}

\medskip

Recall now a well-known result on symmetric  bilinear forms over $\fq_2$.

\begin{prop}\label{proposition:dimension-isotropic}
Let $(V,\B)$ be a symmetric bilinear form of dimension $n$ over $\fq_2$. Denote by $r$ the rank of $\B$. Write $r=2r_0 +\delta$, with $\delta =0$ or $1$, and $r_0 \in \Nat$. 
\begin{enumerate}
\item[$(i)$] If $\B$ is nonalternating, then $(V,\B)$  is isometric to $$
 \overbrace{b(1) \bot \cdots \bot b(1)}^{r_0} \bot \overbrace{\langle 1 \rangle}^{\delta}  \bot \overbrace{\langle 0 \rangle \bot \cdots \bot \langle 0 \rangle}^{n-r} \ \simeq_{iso} \ \overbrace{\langle 1 \rangle \bot \cdots  \bot  \langle 1 \rangle}^{r} \bot \overbrace{\langle 0 \rangle \bot \cdots \bot \langle 0 \rangle}^{n-r} ;$$  
\item[$(ii)$] If $ \B$ is alternating, then $\B$ is isometric to $$ \overbrace{b(0) \bot \cdots \bot b(0)}^{r_0}  \bot  \overbrace{\langle 0 \rangle \bot \cdots \bot \langle 0 \rangle}^{n-r}.$$
\end{enumerate}
Moreover,   $\nu(\B)=n-r+r_0=n-r_0-\delta$.
\end{prop}

When  $(\V,\B)$ is not necessary symmetric, let us introduce the symmetrization $\B^{sym}$ of  $\B$ by $$\B^{sym}(x,y)=\B(x,y)+\B(y,x), \ \ \ \forall x,y \in \V.$$
One has:
\begin{prop}\label{proposition:dimension-isotropic2}
Let $(\V,\B)$ be a bilinear form of dimension $n$ over $\fq_2$. Then $$\nu(\B) \geq n - \lfloor \frac{1}{2} \rk(\B^{sym}) \rfloor - \lfloor \frac{1}{2} \rk(\B) \rfloor.$$
In particular, $\nu(\B) \geq n - \frac{3}{2} \rk(\B)$.
\end{prop}

\begin{proof}
It is easy. Let us start with a maximal totally isotropic subspace $W$ of $(\V,\B^{sym})$.
Then $\B_{|\W}$ is symmetric: indeed, for any two $x,y \in \W$, we get $0=\B^{sym}(x,y)=\B(x,y)+\B(y,x)$, and then $\B(x,y)=\B(y,x)$ (recall that $V$ is defined over $\fq_2$). Hence by Proposition \ref{proposition:dimension-isotropic},  $\B_{|\W}$ has  a  totally isotropic subspace of dimension $\nu(\B_{|\W})=\dim \W - \lfloor \frac{1}{2} \rk(\B_{|\W}) \rfloor$.
As $\dim \W=n-\lfloor \frac{1}{2} \rk(\B^{sym})\rfloor$ (by Proposition \ref{proposition:dimension-isotropic}), one obtains the first assertion. For the second one, it is enough  to note that $\rk(\B^{sym}) \leq 2 \rk(\B)$.
\end{proof}

\subsubsection{Bilinear form over the Kummer radical of the $2$-elementary abelian maximal unramified extension}
Let us start with a totally imaginary number field $\K$. Denote by $n$ the $2$-rank of $\G_\K(2)$, in other words, $n=d_2 \Cl_\K$.

\medskip

Let $V=\langle a_1,\cdots, a_n\rangle (\K^\times)^2 \in \K^\times /(\K^\times)^2$ be the Kummer radical of the $2$-elementary   abelian maximal unramified extension $\K^{ur,2}/\K$. Then $V$ is an $\fq_2$-vector space of dimension~$n$.

As we have seen in section \ref{section:Carlson-Schlank}, for every prime ideal $\p $ of $\O_\K$, the $\p$-valuation of $a_i$
is even, and then $\sqrt{(a_i)}$ as ideal of $\O_\K$ has a sense.

\medskip

For $x\in V$, denote $\K_x:=\K(\sqrt{x})$, and $\aa(x):=\sqrt{(x)} \in \O_\K$. We can now  introduce the bilinear form  $\B_\K$ that  plays a central role in our work.

\begin{defi}
For $a,b \in V$, put: $$\B_\K(a,b)=\left(\frac{\K_a/\K}{\aa(b)}\right)\cdot\sqrt{a} \Bigg/ \sqrt{a}  \in \fq_2,$$
where here we use the additive notation. 
\end{defi}

\begin{rema} The Hilbert symbol between $a$ and $b$ is trivial due to the parity of $v_\q(a)$.
\end{rema}

Of course, we have:

\begin{lemm}
The application $\B_\K: V \times V \rightarrow \fq_2$ is a bilinear form on $V$.
\end{lemm}

\begin{proof} The linearity on the right comes from the linearity of the Artin symbol and the linearity  on the left is an easy observation.
\end{proof}

\begin{rema} \label{remarque:matrix}
If we denote by $\chi_{i}$ a generator of $H^1(\Gal(\K(\sqrt{a_i})/\K))$, then the Gram matrix of the bilinear form $\B_\K$ in the basis $\{a_1(\K^\times)^2,\cdots, a_n(\K^\times)^2\}$ is exactly the matrix  $(\chi_{i} \cup \chi_i \cup \chi_j)_{i,j}$ of the cup-products in $H_{et}^3(\spec \O_\K)$. See Proposition  \ref{proposition:C-S} and Remark \ref{remarque:symbole}.
Hence the bilinear form $\B_\K$ coincides with  the bilinear form $\B_\K^{et}$ on $H_{et}^1(\spec \O_\K)$ defined by $\B_\K^{et}(x,y)=x\cup x\cup y \ \in H^3_{et}(\spec \O_\K)$. 
\end{rema}

The bilinear form $\B_\K$ is not necessarily symmetric, but we will give later some situations where $\B_\K$ is symmetric. 
Let us give now two types of totally isotropic subspaces $\W$ that may appear.

\begin{defi}
The right-radical $\Rad$ of  a bilinear form $\B$ on $\V$  is the subspace  of $\V$ defined by: $\Rad:=\{x\in \V, \ \B(V,x)=0\}$. 
\end{defi}

Of course one  always has 
$\dim \B={\rk \B} + {\dim} {\Rad}$. And,
remark moreover that the restriction of $\B$ at $\Rad$ produces a totally isotropic subspace of $\V$.

\medskip

Let us come back to the bilinear form $\B_\K$  on the Kummer radical of $\K^{ur,2}/\K$.

\begin{prop}
Let $\W:=\langle \varepsilon_1, \cdots, \varepsilon_r \rangle (\K^\times)^2 \subset \V$ be an $\fq_2$-subspace of dimension~$r$, generated by some units $\varepsilon_i \in \O_\K^\times$.
 Then $\W \subset \Rad $, and thus $(\V,\B_\K)$ contains $\W$ as a  totally isotropic subspace of dimension~$r$. 
\end{prop}

\begin{proof}
Indeed, here $\aa(\varepsilon_i)=\O_\K$ for  $i=1,\cdots, r$.
\end{proof}

\begin{prop}
Let $\K=\k(\sqrt{b})$ be a quadratic extension. Suppose that there exist $a_1,\cdots, a_r \in \k$ such that the extensions $\k(\sqrt{a_i})/\k$ are independent and unramified everywhere.  Suppose moreover that $b \notin \langle a_1,\cdots, a_r \rangle (\k^\times)^2$. Then $\W := \langle a_1, \cdots,a_r \rangle (\K^\times)^2 $ is a  totally isotropic subspace of dimension $r$. 
\end{prop}

\begin{proof}
Let $\p \subset \O_\k$ be a prime ideal of $\O_\k$. It is sufficient  to prove that $\displaystyle{\left(\frac{\K_{a_i}/\K}{\p}\right)}$ is trivial. Let us study  all the possibilities.

 $\bullet$ If $\p$ is inert in $\K/\k$, then as $\K(\sqrt{a_i})/\K$ is unramified at $\p$, necessary $\p$ splits in $\K(\sqrt{a_i})/\K$ and then $\displaystyle{\left(\frac{\K_{a_i}/\K}{\p}\right)}$ is trivial.
 
 $\bullet$ If $\p=\P^2$ is ramified in $\K/\k$, then  $\displaystyle{\left(\frac{\K_{a_i}/\K}{\p}\right)=
 \left(\frac{\K_{a_i}/\K}{\P}\right)^2}$ is trivial.
 
$\bullet$  If $\p=\P_1\P_2$ splits, then obviously $\displaystyle{\left(\frac{\K_{a_i}/\K}{\P_1}\right)=\left(\frac{\K_{a_i}/\K}{\P_2}\right)}$,  and then  $\displaystyle{\left(\frac{\K_{a_i}/\K}{\p}\right)}$ is trivial.
\end{proof}

\medskip

It is then natural to define the index of $\K$ as follows:

\begin{defi}
The index  $\nu(\K)$ of $\K$  is the  index of the bilinear form $\B_\K$.
\end{defi}

Of course, if the form $\B_\K$  is non-degenerate, one has: $\nu(\K) \leq \frac{1}{2} d_2 \Cl_\K$.
Thus one says that $\Cl_\K$ is non-degenerate if the form $\B_\K$ is non-degenerate.

\medskip

One can now present the main result of our work:

\begin{theo} \label{maintheorem}
Let $\K/\Q$ be a totally imaginary  number field.  
Then $\G_\K(2)$ has no  uniform quotient of dimension $d> \nu(\K)$. In particular:
\begin{enumerate}
\item[$(i)$] if $\nu(\K) <3$, then the Conjecture \ref{conj2} holds for $\K$ (and $p=2$);
\item[$(ii)$] if $\Cl_\K$ is non-degenerate, then $\G_\K(2)$ has no uniform quotient of dimension $d > \frac{1}{2} d_2 \Cl_\K$.
\end{enumerate}
\end{theo}


\begin{proof} Let $\Gg$ be a non-trivial uniform quotient of $\G_\K(p)$ of dimension $d>0$.  Let $W$ be the Kummer radical of $H^1(\Gg)^\vee $; here $W$ is a subspace of   
 the Kummer radical $\V$ of $\K^{ur,2}/\K$.
As  $d>\nu(\K)$, the space $\W$ is not totally isotropic. Then, one can find $x,y \in H_{et}^1(\Gg) \subset H^1(\X_\K)$ 
such that $x \cup x\cup y \in H^3_{et}(\X_\K)$ is not zero (by Proposition \ref{proposition:C-S}). See also Remark \ref{remarque:matrix}.  And thanks to the stategy developed in Section \ref{section:strategy}, we are done for the first part ot the theorem.

$(i)$: as  $\G(2)$ is FAb,  every non-trivial uniform quotient $\Gg$  of $\G(2)$ should be  of dimension at least $d\geq 3$.

$(ii)$: in this case $\nu(\K) \leq \frac{1}{2} \rk(\B_\K)$.
\end{proof}

We finish this section with the proof of the theorem  presented in the introduction.

$\bullet$  As $\nu(\K) \leq n-\frac{1}{2} \rk(\B_\K)$, see Proposition \ref{bound-nu} and Remark \ref{remarque:matrix}, the assertion $(i)$ can be deduced by Theorem \ref{maintheorem}.

$\bullet$ This is an obvious observation for the small dimensions. In the three cases,  $\nu(\K) \leq n-\frac{1}{2} \rk(\B_\K) <3$.

\subsection{The imaginary quadratic case}

\subsubsection{The context} \label{section:thecontext}

Let us consider an imaginary quadratic field $\K=\Q(\sqrt{D})$, $D \in \Z_{<0}$ square-free. 
Let $p_1,\cdots, p_{k+1}$ be the {\it odd} prime numbers dividing $D$. Let us write  the discriminant $\disc_\K$ of $\K$ by: $\disc_\K=p_0^*\cdot p_1^* \cdots p_{k+1}^*$,
where $p_0^*\in \{1, -4,\pm 8\}$.

The $2$-rank $n$ of $\Cl_\K$ depends on the ramification of  $2$ in $\K/\Q$. Put $\K^{ur,2}$ the  $2$-elementary abelain maximal unramified extension of $\K$:
\begin{enumerate}
\item[$-$] if $2$ is unramified in $\K/\Q$, {\emph i.e.} $p_0^*=1$, then $n=k$ and $V=<p_1^*,\cdots, p_k^*> (\K^\times)^2\subset \K^\times$ is the Kummer radical of $\K^{ur,2}/\K$;
\item[$-$] is $2$ is ramified in $\K/\Q$, {\emph i.e.} $p_0^*=-4$  or $\pm 8$, then $n=k+1$ and $V=<p_1^*,\cdots, p_{k+1}^*> (\K^\times)^2\subset \K^\times$ is the Kummer radical of $\K^{ur,2}/\K$.
\end{enumerate}

We denote by $\SS=\{p_1^*,\cdots, p_{n}^*\}$ the  $\fq_2$-basis of $V$, where here $n=d_2 \Cl_\K$ ($=k$ or $k+1$).

\begin{lemm}

\medskip

$(i)$ For $p^*\neq q^* \in \SS$, one has: $\B_\K(p^*,q^*)=0$ if and only if, $\displaystyle{\left(\frac{p^*}{q}\right)=1}$.

$(ii)$ For $p |D$, put $D_p:=D/p^*$. Then for $p^*\in\SS$, one has: 
 $\displaystyle{\B_\K(p^*,p^*):=\left(\frac{D_p}{p}\right)}$.
 \end{lemm}

\begin{proof}
Obvious.
\end{proof}

Hence the matrix of the bilinear form $\B_\K$ in the basis $\SS$ is a square $n\times n$ R\'edei-matrix type ${\M_\K}=\left( m_{i,j}\right)_{i,j}$,
where $$m_{i,j}=\left\{\begin{array}{ll} \displaystyle{ \left(\frac{p_i^*}{p_j}\right)} & {\rm if \ } i\neq j, \\
 \displaystyle{ \left(\frac{D_{p_i}}{p_i}\right)} & {\rm if \ } i=j.
\end{array}\right.$$ 
Here as usual, one uses the additive notation (the $1$'s are replaced by $0$'s and the $-1$'s by~$1$).

\medskip

\begin{exem} \label{exemple-Boston}
Take $\K=\Q(\sqrt{-4\cdot 3 \cdot 5 \cdot 7 \cdot 13})$. This  quadratic field has a root discriminant $|\disc_\K|^{1/2}= 73. 89 \cdots$, but we dont know actually if $\G_\K(2)$ is finite or not; see the recent works of Boston and Wang \cite{Boston-Wang}.
Take  $\SS=\{-3,-5,-7,-13\}$. Then the Gram matrix of $\B_\K$ in $\SS$ is:
$$\M_\K=\left(\begin{array}{cccc}
1&1&1&0 \\
1&1&1&1 \\
0&1&1&1 \\
0&1&1&0
\end{array}\right).$$
Hence $\rk(\B_\K)=3$ and $\nu(\K) \leq 4- \frac{3}{2} =2.5$. By Theorem \ref{maintheorem}, one concludes  that $\G_\K(2)$ has no non-trivial uniform quotient. 
 By Proposition \ref{proposition:dimension-isotropic2}, remark  that here one has: $\nu(\K)=2$.
\end{exem}

Let us recall at this level a part of  the theorem  of the introduction:

\begin{coro} \label{coro-FM-quadratic}
The Conjecture \ref{conj2} holds when:
\begin{enumerate}
\item[$(i)$] $d_2\Cl_\K= 5$ and $\B_\K$ is non-degenerate;
\item[$(ii)$] $d_2\Cl_\K=4$ and $\rk(\B_\K) \geq 3$;
\item[$(iii)$] $d_2\Cl_\K=3$ and $\rk(\B_\K) >0$.
\end{enumerate}
\end{coro}

Remark that $(iii)$ is an extension of corollary \ref{coro-exemple}.

\subsubsection{Symmetric bilinear forms. Examples}

Let us conserve the context of the previous section \ref{section:thecontext}. Then,  thanks to the quadratic  reciprocity law, one gets: 

\begin{prop} \label{prop:congruence} The bilinear form $\B_\K: V \times V \rightarrow \fq_2$ is symmetric, if and only if, there is at most one prime $p \equiv   3 (\mod 4)$ dividing $D$.
\end{prop}

\begin{proof} Obvious.
\end{proof}



Let us give some examples (the computations have been done by using PARI/GP \cite{pari}).

\begin{exem}
Take $k+1$ prime numbers $p_1,\cdots, p_{k+1}$, such that
\begin{enumerate}
\item[$\bullet$] $p_1\equiv \cdots p_{k} \equiv 1({\rm mod} \ 4)$ and $p_{k+1} \equiv 3 ({\rm mod}  \ 4)$;
\item[$\bullet$] for $1\leq i <  j \leq k$,  $\displaystyle{\left(\frac{p_i}{p_j}\right)=1}$;
\item[$\bullet$] for $i=1,\cdots, k$, $\displaystyle{\left(\frac{p_i}{p_{k+1}}\right)=-1}$
\end{enumerate}
Put $\K=\Q(\sqrt{-p_1 \cdots p_{k+1}})$.
In this case the matrix of  the bilinear form $\B_\K$ in the basis $(p_i)_{1 \leq k}$ is the identity matrix of dimension $k \times k$ and,   $\nu(\K)=\lfloor \frac{k}{2} \rfloor$. 
Hence, $\G_\K(p)$ has no uniform quotient of dimension at least $\lfloor \frac{k}{2} \rfloor +1$.

In particular, if we  choose an integer $t>0$ such that \ $\sqrt{k+1}\geq t \geq \sqrt{\lfloor \frac{k}{2} \rfloor +2}$, then there is no quotient   of $\G_\K(2)$ onto $\Sl_t^1(\Z_2)$. (If $t > \sqrt{k+1}$, it is obvious.)

\medskip

Here are some more examples. For $\K_1=\Q(\sqrt{-5\cdot 29 \cdot 109 \cdot 281 \cdot 349 \cdot 47})$, $\G_{\K_1}(2)$ has no uniform quotient; here $\Cl(\K_1) \simeq   (\Z/2\Z)^5$.

Take  $\K_2=\Q(\sqrt{-5\cdot 29 \cdot 109 \cdot 281 \cdot 349 \cdot 1601 \cdot 1889 \cdot 5581 \cdot 3847})$; here $\Cl(\K_2) \simeq (\Z/2\Z)^8 $. Then $\G_{\K_2}(2)$ has no uniform quotient of dimension at least $5$. In particular, there is 
no unramified extension of~$\K_2$ with Galois group isomorphic to $\Sl_3^1(\Z_2)$.

\end{exem}

\medskip

\begin{exem}
Take $2m+1$ prime numbers $p_1,\cdots, p_{2m+1}$, such that
\begin{enumerate}
\item[$\bullet$] $p_1\equiv \cdots p_{2m} \equiv 1({\rm mod} \ 4)$ and $p_{2m+1} \equiv 3 ({\rm mod}  \ 4)$;
\item[$\bullet$] $\displaystyle{\left(\frac{p_1}{p_2}\right)=  \left(\frac{p_3}{p_4}\right) = \cdots = \left(\frac{p_{2m-1}}{p_{2m}}\right) = -1}$;
\item[$\bullet$] for the other indices $1\leq i <j \leq 2m$, $\displaystyle{\left(\frac{p_i}{p_j}\right)= 1}$;
\item[$\bullet$] for $i=1,\cdots, 2m$, $\displaystyle{\left(\frac{p_i}{p_{2m+1}}\right)=-1}$
\end{enumerate}
Put $\K=\Q(\sqrt{-p_1\cdots p_{2m+1}})$.
In this case the bilinear form $\B_\K$ is nondegenerate and alternating, then isometric to  $ \displaystyle{\overbrace{b(0) \bot \cdots \bot b(0)}^{m} }$. Hence,  $\nu(\K)=m$, and $\G_\K(p)$ has no uniform quotient of dimension at least $m +1$.

For example, for  $\K=\Q(\sqrt{-5\cdot 13 \cdot 29 \cdot 61 \cdot 1049 \cdot  1301 \cdot 743})$, $\G_\K(2)$ has no uniform quotient of dimension at least $4$.
\end{exem}

\subsubsection{Relation with the $4$-rank of the Class group - Density estimations}

The study of the $4$-rank of the class group of quadratic number fields started with the work of R\'edei \cite{Redei}  (see also \cite{Redei-Reichardt}). Since, many authors
have contribued to its extensions, generalizations and applications. Let us cite an article of  Lemmermeyer \cite{Lemmermeyer} where one can find a 
large litterature about the question. See also a nice paper of Stevenhagen \cite{Stevenhagen},  and  the  work of Gerth \cite{Gerth} and  Fouvry-Kl\"uners  \cite{Fouvry-Klueners} concerning the density question.

\begin{defi}
Let $\K$ be a number field, define by $\R_{\K,4}$ the $4$-rank of $\K$ as follows: $$\R_{\K,4}:= \dim_{\fq_2} \Cl_\K[4]/\Cl_\K[2],$$
where $\Cl_\K[m]=\{c \in \Cl_\K,c^m=1\}$.
\end{defi}

Let us conserve the context and the notations of the section \ref{section:thecontext}: here  $\K=\Q(\sqrt{D})$ is  an imaginary  quadratic  field of discrimant $\disc_\K$, $D \in \Z_{<0}$ square-free. Denote by $\{q_1,\cdots q_{n+1}\}$ the set of prime numbers that ramify in $\K/\Q$; $d_2 \Cl_\K=n$. Here we can take $q_i=p_i$ for $1\leq i \leq  n$, and $q_n=p_{k+1}$ or $q_n=2$ following the ramification at $2$. Then,
consider the R\'edei matrix $\M_\K'=(m_{i,j})_{i,j}$ of size $(n+1)\times (n+1)$ with coefficients in $\fq_2$, 
where  $$m_{i,j}=\left\{\begin{array}{ll} \displaystyle{ \left(\frac{q_i^*}{q_j}\right)} & {\rm if \ } i\neq j, \\
 \displaystyle{ \left(\frac{D_{q_i}}{q_i}\right)} & {\rm if \ } i=j.
\end{array}\right.$$ 
It is not difficult to see that the sum of the rows is zero, hence the rank of $\M'_\K$ is smaller than $n$.

\begin{theo}[R\'edei] \label{theorem:Redei} Let $\K$ be  an imaginary quadratic number field. Then $ \R_{\K,4}=d_2 \Cl_\K-\rk(\M'_\K)$.
\end{theo}

\begin{rema}
The strategy of R\'edei is to construct for every couple $(D_1,D_2)$ "of second kind", a degree $4$ cyclic  unramified extension of $\K$. Here  to be of  second kind  
means that $\disc_\K=D_1D_2$, where $D_i$ are fundamental discriminants  such that $\left(\frac{D_1}{p_2}\right)= \left(\frac{D_2}{p_1}\right)=1$, for every prime $p_i|D_i$,  $i=1,2$.
And clearly, this condition corresponds exactly to the existence of orthogonal subspaces $W_i$  of the Kummer radical $\V$, $i=1,2$, generated 
by the $p_i^*$, for all $p_i|D_i$: $\B_\K(W_1,W_2)=\B_\K(W_2,W_1)=\{0\}$. Such orthogonal subspaces allow us to construct totally isotropic subspaces. 
And then, the larger the $4$-rank of $\Cl_\K$, the larger $\nu(\K)$ must be. 
\end{rema}

Consider now the matrix $\M''_\K$ obtained from $\M'_\K$ after missing the last row.
Remark here that the matrix $\M_\K$ is a submatrix of the R\'edei matrix $\M''_\K$:

$$\M''_\K=\left(\begin{array}{cl}
\begin{array}{|c|} \hline \\ \M_\K \\  \\ \hline \end{array} 
& \begin{array}{c}  * \\ \vdots \\ * \end{array} \end{array} \right)$$

Hence, $\rk(\M_\K) +1 \geq  \rk(\M'_\K) \geq \rk(\M_\K)$.   Remark that in example \ref{exemple-Boston}, $\rk(\M_\K)=3$ and $\rk(\M'_\K)=4$.
But  sometimes one has $\rk(\M'_\K)=\rk(\M_\k)$,  as for example:
\begin{enumerate}
\item[$(A)$:] when: $p_0=1$ (the set  of  primes $p_i \equiv 3 ({\rm mod} \ 4)$ is odd);
\item[$(B)$:] or, when $\B_\K$ is non-degenerate.
\end{enumerate}
 
For situation $(A)$,  it suffices to note that the sum of the columns  is zero (thanks to the properties of the Legendre symbol). 

\medskip

From now on we follow the  work of Gerth \cite{Gerth}. Denote by $\Ff$ the set of imaginary quadratic number fields.
For  $0 \leq r  \leq n$ and $X\geq 0$, put
$$\S_{X}=\big\{ \K \in \Ff, \ |\disc_\K|\leq X \big\},$$
$$\S_{n,X}=\big\{\K\in \S_{X},  \ d_2\Cl_\K=n \big\}, \ \ \S_{n,r,X}=\big\{\K \in \S_{n,X}, \  \R_{\K,4}=r \big\}.$$
Denote also $$\A_{X}=\big\{\K\in \S_X, {\rm \ satisfying } \ (A) \big\}$$
$$ \A_{n,X}=\big\{ \K \in \A_{X}, \ d_2 \Cl_\K=n\big\}, \ \  \A_{n,r,X}=\big\{ \K \in \A_{n,X},
\  \R_{\K,4}=r\big\}.$$
One has the following density theorem due to Gerth:

\begin{theo}[Gerth \cite{Gerth}]
The limits $\displaystyle{\lim_{X\rightarrow \infty} \frac{|\A_{n,r,X}|}{|\A_{n,X}|} }$ and
$\displaystyle{\lim_{X\rightarrow \infty} \frac{|\S_{n,r,X}|}{|\S_{n,X}|} }$  exist and are equal. Denote by  $d_{n,r}$ this quantity.
 Then  $d_{n,r}$ is   explicit and,  $$d_{\infty,r}:=\lim_{n \rightarrow \infty} d_{n,r}= \frac{2^{-r^2} \prod_{k=1}^\infty(1-2^{-k})}{\prod_{k=1}^r(1-2^{-k})}.$$
\end{theo}

\medskip 

Recall also the following quantities introduced at the beginning of our work:
$$\FM_{n,X}^{(d)}:=\{ \K \in \S_{n,X},  \ \G_\K(2) {\rm \ has \ no \ uniform \ quotient \ of \ dimension} > d\},$$
 $$ \FM_{n,X}:=\{ \K \in \S_{n,X}, \  {\rm Conjecture \ \ref{conj2} \ holds  \ for \ }\K\},$$
 and the limits:
$$\FM_{n}:=\liminf_{X \rightarrow + \infty} \frac{\# \FM_{n,X}}{\#\S_{n,X}}, \ \ \ \FM_n^{(d)}:= \liminf_{X\rightarrow + \infty} \frac{\# \FM_{n,X}^{(d)}}{\#\S_{n,X}}.$$

After combining all our observations, we obtain (see also Corollary \ref{coro-intro1}):
\begin{coro}
For $d\leq n$, one has $$\FM_n^{(d)} \geq d_{n,0}+d_{n,1}+\cdots + d_{n,2d-n-1}.$$
In particular:
\begin{enumerate}
\item[$(i)$] $ \FM_{3} \geq .992187 $;
\item[$(ii)$] $\FM_{4} \geq .874268$,  $\FM_4^{(4)} \geq .999695$;
\item[$(iii)$] $\FM_5 \geq .331299$,  $\FM_5^{(4)} \geq  .990624$, $\FM_5^{(5)} \geq  .9999943$;
\item[$(iv)$] $\FM_{6}^{(4)} \geq .867183$, $\FM_6^{(5)} \geq .999255$, $\FM_6^{(6)} \geq  1-5.2 \cdot 10^{-8}$;
\item[$(v)$] for all $d\geq 3$, $\FM_d^{(1+d/2)} \geq .866364$, $\FM_d^{(2+d/2)} \geq  .999953$.
\end{enumerate}
\end{coro}

\begin{proof}
As noted by Gerth in \cite{Gerth}, the dominating set in the density computation is the set $\A_{n,X}$  of imaginary quadratric number fields $\K=\Q(\sqrt{D})$ satisfying $(A)$.
But for $\K$ in $\A_{n,X}$, one has ${\rk(\B_\K)}={\rk(\M_\K)}= {n}-\R_{\K,4}$. Hence for 
$\K \in \A_{n,X,r}$, by Proposition \ref{bound-nu} $$\nu(\K) \leq n -\frac{1}{2} \big(n - \R_{\K,4}\big)=\frac{1}{2}\big(n+ \R_{\K,4}\big).$$
Now one uses Corollary \ref{coro-FM-quadratic}. Or equivalently,  one sees that Conjecture \ref{conj2} holds when  $3 >  \frac{1}{2}\big( n + \R_{\K,4}\big)$, {\emph i.e.}, when
$\R_{\K,4} < 6-n$.
More generaly, $\G_\K(2)$ has no uniform quotient of dimension  $d$   when
$\R_{\K,4} < 2d -n$. In particular, $$\FM_n^{(d)} \geq d_{n,0}+d_{n,1}+\cdots + d_{n,2d-n-1}.$$
Now one uses the estimates of Gerth in \cite{Gerth}, to obtain: \begin{enumerate}
\item[$(i$)]$\FM_{3} \geq d_{3,0}+d_{3,1}+d_{3,2} \simeq 0.992187 \cdots$
\item[$(ii)$] $\FM_4 \geq d_{4,0} + d_{4,1} \simeq 0.874268 \cdots $, $\FM_4^{(4)} \geq d_{4,0} + d_{4,1}+d_{4,2}+d_{4,3}  \simeq .999695 \cdots$, 
\item[$(iii)$] $ \FM_5 \geq d_{5,0} \simeq 0.331299 \cdots $, $\FM_5^{(4)} \geq d_{5,0}+d_{5,1}+d_{5,2} \simeq .990624 \cdots$, $\FM_5^{(5)} \geq d_{5,0}+d_{5,1}+d_{5,2}+d_{5,3}+d_{5,4} \simeq .9999943 \cdots$,
\item[$(iv)$] $\FM_6^{(4)} \geq d_{6,0}+d_{6,1} \simeq  0.867183 \cdots$, $\FM_6^{(5)} \geq d_{6,0}+d_{6,1}+d_{6,2}+d_{6,3} \simeq .999255 \cdots$, $\FM_6^{(6)} \geq 1- d_{6,6} \simeq 1-5.2 \cdot 10^{-8}$,
\item[$(v)$]   $\FM_d^{(1+d/2)} \geq d_{\infty,0}+ d_{\infty,1} \simeq .866364\cdots $,  $\FM_d^{(2+d/2)} \geq d_{\infty,0}+d_{\infty,1}+d_{\infty,2}+d_{\infty,3} \simeq .999953 \cdots$.
\end{enumerate}
\end{proof}

In the spirit of the Cohen-Lenstra heuristics, the work of Gerth has been improved by Fouvry-Kl\"uners \cite{Fouvry-Klueners}.
This work allows us to give a more general density estimation as announced in Introduction.
Recall $$\FM_{X}^{[i]}:=\{\K \in \S_{X},  \ \G_\K(2) {\rm \ has \ no \ uniform \ quotient \ of \ dimension} > i+ \frac{1}{2}d_2 \Cl_\K\}$$
and $$\FM^{[i]}:= \liminf_{X\rightarrow + \infty} \frac{\# \FM^{[i]}_X}{\# \S_{X}}.$$
Our work allows us to obtain (see Corollary \ref{coro-intro2}):
\begin{coro} 
 For $i\geq 1$, one has: $$\FM^{[i]} \geq d_{\infty,0}+ d_{\infty,1}+ \cdots + d_{\infty,2i-2} .$$
 In particular, $$\FM^{[1]} \geq  .288788, \   \ \FM^{[2]} \geq . 994714, \ {\rm and } \ \FM^{[3]} \geq 1-9.7\cdot  10^{-8}.$$
\end{coro}

\begin{proof}
By Fouvry-Kl\"uners \cite{Fouvry-Klueners}, the density of imaginary quadratic fields for which $\R_{\K,4}=r$, is equal to  $d_{\infty,r}$.
Remind of that for $\K \in \Ff$: $\rg(\M_\K) \geq \rg(\M_\K') -1 $. Then thanks to Proposition \ref{bound-nu} and Theorem \ref{theorem:Redei}, we get $$\nu(\K) \leq \frac{1}{2} d_2 \Cl_\K  + \frac{1}{2} +\frac{1}{2}  \R_{\K,4}.$$ 
Putting this fact together with Theorem \ref{maintheorem}, we obtain that $\G_\K(2)$ has no uniform quotient
of dimension $d > \frac{1}{2} d_2 \Cl_\K  + \frac{1}{2} +\frac{1}{2}  \R_{\K,4}$. Then  for $i\geq 1$, the proportion of the fields $\K$ in $\FM^{[i]}$ is at least the proportion of $\K \in \Ff$ for which $\R_{\K,4} < 2i -1$, hence at least $d_{\infty,0}+ d_{\infty,1}+ \cdots + d_{\infty,2i-2}$ by \cite{Fouvry-Klueners}.
To conclude:
$$\FM^{[1]} \geq d_{\infty,0} \simeq .288788\cdots$$ $$ \FM^{[2]} \geq d_{\infty,0} + d_{\infty,1}+ d_{\infty,2}  \simeq .994714\cdots$$
$$ \FM^{[3]} \geq  d_{\infty,0} + d_{\infty,1}+d_{\infty,2} +d_{\infty,3} +d_{\infty,4} \simeq 1-9.7\cdot  10^{-8}.$$
\end{proof}

\



\end{document}